\documentclass[a4paper]{amsart}

\usepackage{amsmath, amssymb, amsthm, comment}
\usepackage{a4wide,rotating,enumitem, bbold}

\usepackage{mathabx}

\newcommand\F{\mathcal{F}}
\newcommand\1{\mathbb{1}}

\newcommand\cop{\Delta}

\usepackage[sans]{dsfont}

\theoremstyle{plain}
\newtheorem{theorem}{Theorem} 

\newtheorem{lem}[theorem]{Lemma}

\theoremstyle{remark}
\newtheorem{remark}{Remark}

\DeclareMathOperator\Irr{Irr}

\DeclareMathOperator\Mor{Mor}
\DeclareMathOperator\Tr{Tr}

\newcommand\sm{\setminus}

\newcommand{\C}{\mathbb{C}}
\newcommand{\X}{\mathbb{X}}

\newcommand{\ellinf}{\ell^\infty}

\newcommand{\inv}{^{-1}}
\newcommand\conj{\overline}

\newcommand{\G}{\mathbb{G}}
\newcommand{\Gd}{\mathbb{\Gamma}}

\newcommand{\ot}{\otimes}
\newcommand{\id}{\mathrm{id}}

\begin{document}

\title{On free actions of discrete quantum groups}
\author{Pekka Salmi}
\address{Research Unit of Mathematical Sciences, University of Oulu, PL 8000, FI-90014, Finland}
\email{pekka.salmi@oulu.fi}

\begin{abstract}
R. Ellis showed in 1960 that every discrete group acts freely on its
Stone--\v Cech compactification.
We extend this result to discrete quantum groups with low duals. 
The method of proof is different from the earlier proofs
in the classical case, using the definition of
freeness given by D. A. Ellwood
in the setting of noncommutative geometry.
\end{abstract}

\date {\today}
\maketitle

\section{Introduction}

An action of a group $G$ on a set $X$ is free if
no element in $G$ except the identity $e$ has any fixed points in $X$,
i.e.\ $sx\ne x$ for every $s\in G\sm\{e\}$  and $x\in X$.
Clearly, every group acts freely on itself.
The action of a discrete group $G$ on itself extends to a continuous action
of $G$ on its Stone--\v Cech compactification $\beta G$.  
R. Ellis \cite{ellis} showed that this action is always
free, there-by showing that every discrete group acts freely
on some compact space.

D. A. Ellwood  \cite{ellwood} extended the notion of free action to the setting
of noncommutative geometry. In particular, it is applicable to
actions of locally compact quantum groups on noncommutative spaces.
Several authors have studied free actions of compact quantum groups;
see for example \cite{de-commer-yamashita,nes-yam,sch-wag}.
In particular, P. F. Baum, K. De Commer and P. M. Hajac \cite{baum-et-al} 
gave three different characterisations of free actions in the case of  
a compact quantum group acting on a compact noncommutative space.
On the other hand, free actions of discrete quantum groups  
have not garnered much attention. Recently, using a different noncommutative extension of freeness, 
B.~Anderson-Sackaney and R. Vergnioux \cite{asv} showed that the action 
of a discrete quantum group $\Gd$ on $\ellinf(\Gd)$ is 
free if and only if $\ellinf(\Gd)$ is commutative. 
Using Ellwood's definition of freeness, the situation is different, 
as we shall show.

The main result of this paper
is that a discrete quantum group with a low dual acts
freely on its universal compactification.
By universal compactification we mean the natural noncommutative
analogue of the Stone--\v Cech compactification:
if a noncommutative space is associated to a C*-algebra $A$, then
its universal compactification is associated to the 
multiplier algebra $M(A)$. 
Note that when $A = C_0(X)$ is commutative, then $M(A)  \cong  C(\beta X)$. 
The action of a discrete quantum group $\Gd$ extends naturally to
its universal compactification, that is, the noncommutative space associated to $\ellinf(\Gd)$. 
We show that this action is free under the assumption
that the discrete quantum group has a low dual. 
A discrete quantum group $\Gd$ is said to have a low dual if its 
dual $\G$, which is a compact quantum group, 
is such that the dimensions of all irreducible representations
of $\G$ are universally bounded by some constant.
In the classical case such compact groups are virtually abelian.
In the case of quantum groups, there are more diverse 
examples of compact quantum groups with bounded dimension on the
irreducible representations: as noted in \cite[Remark 1.6]{dsv}, 
examples can be constructed using bicrossed products and 
as noted in \cite{kra-sol}, examples arise also naturally in 
noncommutative geometry \cite{ban-mar, haj-mas}.
Note that J. Krajczok and P. M. So\l tan \cite{kra-sol}
proved that any  compact quantum group with this property 
is of Kac type.

\section{Notation and preliminaries}

A \emph{compact quantum group} $\G$ consists of a unital C*-algebra $C(\G)$
and a comultiplication $\cop: C(\G)\to C(\G)\ot C(\G)$
that is coassociative and satisfies the cancellation property
(see for example \cite{woronowicz:compact}, \cite{maes-van-daele:compact} 
or \cite{neshveyev-tuset:book} for details).
In the case when $C(\G)$ is commutative, the C*-algebra $C(\G)$
is isomorphic to the C*-algebra of continuous functions
on a compact group, which explains the notation.

Next we fix our notation concerning the representation theory of 
compact quantum groups. We refer the reader to Chapters 1 and 2 of 
\cite{neshveyev-tuset:book} for more details.
A \emph{representation} of a compact quantum group $\G$
on a finite-dimensional Hilbert space $H$ is an invertible operator
$u\in  B(H) \ot C(\G)$  such that
$(\id \ot \cop)u = u_{12}u_{13}$, where we use the standard leg-numbering
notation. Except for the contragredient representation, we are only interested in 
representations which are unitary.
Given a representation $u$, the underlying Hilbert space
is  denoted by $H_u$ and the identity operator on $H_u$ is denoted by $1_u$. 
If $u$ and $v$ are two representations
of $\G$, then an \emph{intertwiner} between $u$ and $v$
is a linear map $T: H_u\to H_v$ such that
\[
(T\ot 1)u = v(T\ot 1). 
\]
The set of intertwiners between $u$ and $v$ is denoted by
$\Mor(u, v)$. A representation $u$ is a
\emph{subrepresentation} of $v$, denoted by $u\subset v$, if
$\Mor(u, v)$ contains an isometry. 
The \emph{tensor product of representations} $u$ and $v$ is defined by
\[
u\ot v = u_{13}v_{23}.
\]

For a Hilbert space $H$, denote the conjugate Hilbert space by $\overline{H}$
and define $J: B(H)\to B(\overline{H})$ by 
$J(A)\overline{\xi} = \overline{A^* \xi}$.
The \emph{contragredient representation} of a representation $u$ is defined by 
$u^c = (J\ot \id)(u\inv)$. 
Following the notation of \cite[section 1.4]{neshveyev-tuset:book}, let $\rho_u$ denote the 
unique invertible positive operator in $\Mor(u, u^{cc})$ such that 
$\Tr(\cdot \rho_u) = \Tr(\cdot\rho_{u}\inv)$ on $\Mor(u, u)$, 
where $\Tr$ denotes the trace on $B(H_u)$. 
The contragredient representation is not necessarily unitary, 
but the \emph{conjugate representation} defined by 
\[
\overline{u} = (J(\rho_u)^{1/2}\ot 1)u^c (J(\rho_u)^{-1/2}\ot 1)
\]
is unitary. From now on, ``representation'' refers to a unitary
representation on a finite-dimensional Hilbert space.

A representation $u$ is \emph{irreducible} if $\Mor(u, u) = \C$.
Choose a representative from each equivalence class of irreducible
representations and denote the collection
by $\Irr(\G)$. We identify an equivalence class with its
representative. The trivial representation is denoted by $\1$. 

For $u\in \Irr(\G)$, fix an orthonormal base $\{e_i\}$ of $H_u$ and 
define 
\[
t_u = \sum_i \rho_u^{1/2} e_i\ot \conj{e_i}\qquad\text{and}\qquad
s_u = \sum_i \conj{e_i}\ot \rho_u^{-1/2} e_i.
\]
Then $t_u\in \Mor(\1, u\ot \conj{u})$, $s_u\in\Mor(\1, \conj{u}\ot u)$  and 
they satisfy the conjugate equations 
\begin{equation} \label{eq:conjugate}
(1_ u\ot s_ u^*)(t_ u\ot 1_ u) = 1_u \qquad\text{and}\qquad
(s_u^*\ot 1_{\conj u})(1_{\conj u}\ot t_u) = 1_{\conj u}.
\end{equation}
Define $j: B(H_{\conj u})\to B(H_u)$ by
\[
j(A) = \rho_u^{1/2} J\inv (A) \rho_u^{-1/2},
\]
so that 
\begin{equation} \label{eq:j}
(j(A)\ot 1_{\conj u}) t_u = (1_u\ot A) t_u
\end{equation}
for every $A\in B(H_{\conj u})$.

We will work with discrete quantum groups, which may be considered
as duals of compact quantum groups, see
\cite{woronowicz:compact, maes-van-daele:compact, neshveyev-tuset:book}. 
A discrete quantum group $\Gd$ is associated with a possibly
nonunital C*-algebra $c_0(\Gd)$ and a comultiplication
$\cop:c_0(\Gd)\to M(c_0(\Gd)\ot c_0(\Gd))$ where $M$ refers to
the multiplier algebra.
We write $\ellinf(\Gd) = M(c_0(\Gd))$. Note that
in the case of an actual discrete group $\Gamma$, the C*-algebra
$\ellinf(\Gamma)$ is isomorphic to the C*-algebra $C(\beta \Gamma)$
of continuous functions on the Stone--\v Cech compactification
of $\Gamma$. Therefore, the noncommutative space associated to 
$\ellinf(\Gd)$ may be considered as the
noncommutative analogue of the Stone--\v Cech compactification
of $\Gd$, which we call the \emph{universal compactification} of $\Gd$ 
(other compactifications being associated to the
unital C*-subalgebras $A$ of $\ellinf(\Gd)$ such that
$c_0(\Gd)\subset A \subset \ellinf(\Gd)$).

We may describe $c_0(\Gd)$ as a $c_0$-direct sum of finite-dimensional
C*-algebras:
\[
c_0(\Gd) = \bigoplus_{ u\in \Irr(\G)} B(H_ u).
\]
Then $\ellinf(\Gd)$ is isomorphic to the corresponding
direct product:
\[
\ellinf(\Gd) = \prod_{ u\in \Irr(\G)} B(H_ u).
\]
Denote the projection from $\ellinf(\Gd)$ onto $B(H_ u)$ by $p_ u$.
We may view $p_ u$ as an element of $\ellinf(\Gd)$
and $B(H_ u)\subset\ellinf(\Gd)$.
We shall say that an element $a\in \ellinf(\Gd)$ is
\emph{supported} by $I\subset \Irr(\G)$ if
$a p_u = 0$ for every $ u\notin I$.

The comultiplication of $\Gd$ is defined by 
\[
\cop(a) = \sum_{v\in \Irr(\G)}
\sum_{w\in \Irr(\G)} \sum_{T: u\subset v\ot w} T a p_u T^*
\]
where the last  sum runs through 
the isometric intertwiners $T:H_ u\to H_v\ot H_w$
in the decomposition of $v\ot w$ into
irreducible subrepresentations $ u$, accounting for multiplicities. See for example \cite{yamagami},
or for the derivation of an analogous formula, see \cite{kra-sol}. 
The counit of $\Gd$ is the linear map $\epsilon: c_0(\Gd) \to \C$
such that $a p_\1 = \epsilon(a) p_\1$ for every $a\in c_0(\Gd)$.
Note that the antipode $S$ of the
discrete quantum group $\Gd$ is defined by
$S(a)p_{u} = j(ap_{\conj u})$ where $a\in c_0(\Gd)$ is such that
$S(a)\in c_0(\Gd)$.
We will not use this fact explicitly, but will work directly with the
maps $j$ and the identity \eqref{eq:j}.

D. A. Ellwood \cite{ellwood} extended the notion of free action in
such a way that it is applicable to actions of (locally compact)
quantum groups on noncommutative spaces.
We recall this definition in the context of discrete quantum groups
acting on compact noncommutative spaces and will use this notion of 
freeness throughout the paper. A compact noncommutative
space $\X$ is a virtual object associated to a unital C*-algebra
$C(\X)$, which may 
be viewed as the C*-algebra of continuous functions on $\X$.
An action of a discrete quantum group $\Gd$ on a compact
noncommutative space $\X$ is
a unital $*$-homomorphism $\alpha: C(\X) \to M(c_0(\Gd)\ot C(\X))$
such that $(\cop\ot\id)\alpha = (\id\ot \alpha)\alpha$
and the linear span of
$\alpha(C(\X))(c_0(\Gd)\ot 1)$ is dense 
in $c_0(\Gd)\ot C(\X)$.
Following D. A. Ellwood \cite{ellwood}, an action
$\alpha$  is \emph{free} if the image of 
\[
\Phi: C(\X)\ot_{alg} C(\X) \to M(c_0(\Gd)\ot C(\X)),
\quad \Phi(x\ot y) = \alpha(x)(1\ot y),
\]
is strictly dense in $M(c_0(\Gd)\ot C(\X))$. 
Recall that the strict topology on a multiplier algebra $M(A)$
of a C*-algebra $A$ is defined by all the seminorms
$x\mapsto \|xa\|$ and $x\mapsto \|ax\|$ where $a\in A$.

In this paper, we consider the action of a discrete quantum group $\Gd$ on 
its universal compactification. This action is obtained as follows.
The comultiplication of $\Gd$ is a nondegenerate $*$-homomorphism
$\cop:c_0(\Gd)\to M(c_0(\Gd)\ot c_0(\Gd))$, so it has a
strictly continuous extension
$\cop: M(c_0(\Gd)) \to M(c_0(\Gd)\ot c_0(\Gd))$. 
Recall that $M(c_0(\Gd)) = \ellinf(\Gd)$,
and note that as normed spaces $M(c_0(\Gd)\ot c_0(\Gd))$ and
$M(c_0(\Gd)\ot \ellinf(\Gd))$ are the same, although their strict
topologies are different.  
Therefore, $\cop:\ellinf(\Gd)\to M(c_0(\Gd)\ot \ellinf(\Gd))$
defines the (left) action of $\Gd$ on its universal compactification
$\ellinf(\Gd)$.
In the case of an actual discrete group $\Gamma$, this action
is the left action of $\Gamma$ on its Stone--\v Cech compactification
$\beta \Gamma$. 

\section{A new look at the Ellis theorem}

In this section we prove the Ellis theorem that every discrete group acts freely on its Stone--\v Cech compactification using the characterisation of freeness due to Ellwood. This gives a new proof for this result 
(which is not hard to prove anyway), but the main purpose is to motivate the proof in the case of discrete quantum groups with low duals. 

The (left) action of a discrete group $\Gamma$ on its Stone--\v Cech compactification $\beta \Gamma$
is given by the comultiplication
\[
\cop : \ellinf(\Gamma) \to \ellinf(\Gamma, \ellinf(\Gamma)) \cong M(c_0(\Gamma) \ot \ellinf(\Gamma)),
\]
where $\ellinf(\Gamma, \ellinf(\Gamma))$ denotes 
the space of bounded $\ellinf(\Gamma)$-valued functions on $\Gamma$.
The strict topology on  $M(c_0(\Gamma) \ot \ellinf(\Gamma))$ coincides with the topology 
of uniform convergence on finite sets when 
$M(c_0(\Gamma) \ot \ellinf(\Gamma))$ is identified with 
$\ellinf(\Gamma, \ellinf(\Gamma))$.

The following lemma is due to Kat\v etov \cite{katetov}; see also \cite[Lemma 3.33]{hindman-strauss}. Note that 
the classical Ellis theorem follows easily from this lemma;
see for example \cite[Lemma 6.28]{hindman-strauss}
concerning the more general case of semigroups with cancellation properties.
We will extend Kat\v etov's lemma to multivalued functions in the next section and apply that in the case of discrete quantum groups.

\begin{lem}\label{lemma:katetov}
  Let $X$ be a set and $f:X\to X$ a function with no fixed points.
  Then there is a partition $X = A_1\sqcup A_2\sqcup A_3$
  such that $f(A_i)\cap A_i = \emptyset$ for every $i = 1, 2, 3$.
\end{lem}

Repeated application of the preceding lemma gives the following result.

\begin{lem} \label{lemma:partition}
  Let $X$ be a set and $f_k:X\to X$, $k=1, 2, \ldots, n$
  functions with no fixed points.
  Then there is a partition $X = \bigsqcup_{i=1}^{3^n} A_i$
  such that $f_k(A_i)\cap A_i = \emptyset$ for every
  $k= 1, 2, \ldots, n$ and $i = 1, 2, \ldots, 3^n$.
\end{lem}

The next theorem is a reformulation of the Ellis theorem 
using the characterisation of freeness due to Ellwood.  
It is known to be true by combining the Ellis theorem with Theorem 2.9 of Ellwood \cite{ellwood}.
However, we give a direct, constructive proof of the result, which
motivates our argument in the case of discrete quantum groups.

\begin{theorem}
  Let $\Gamma$ be a discrete group. 
  The image of the map
  \[
  \Phi: (f, g) \mapsto \cop(f)(1 \ot g) :
  \ellinf(\Gamma)\ot_{alg} \ellinf(\Gamma) \to \ellinf(\Gamma, \ellinf(\Gamma))
  \]
  is dense in the topology of uniform convergence on finite sets.
\end{theorem}

\begin{proof}
In this proof $1_A$ denotes the characteristic function of $A\subset \Gamma$ and 
write $1_v$ for $1_{\{v\}}$ if $v\in \Gamma$; this is consistent with the earlier notation.
Denote the closure of the image of $\Phi$ by $X$. 
It is enough to show that $1_v\ot 1 \in X$, 
as then also the functions of the form $1_v \ot h$ with $h\in \ellinf(\Gamma)$ are in $X$ 
and then also the functions of the form $k\ot h$ with $k\in c_0(\Gamma)$ and  $h\in \ellinf(\Gamma)$ are in $X$.

For any finite set $F\subset \Gamma$, apply Lemma~\ref{lemma:partition}
to obtain a finite partition $\Gamma = \bigsqcup_{i\in I_F} A_i$
such that
\[
sA_i \cap A_i = \emptyset 
\]
for every $s\ne e$ in $F$ and $i\in I_F$. 
Define
\[
g_F = \sum_{i\in {I_F}} \cop(1_{vA_i})(1\ot 1_{A_i}).
\]
Then for every $s, t\in \Gamma$ we have
\[
g_F(s, t) = \sum_{i\in {I_F}} 1_{s\inv vA_i\cap A_i}(t). 
\]

Next we shall show that $g_F \to 1_v\ot 1$ uniformly on finite sets 
as $F uparrow \Gamma$. 

Let $K\subset \Gamma$ be finite and let $F$ be any finite subset of $\Gamma$
such that $K\inv v \subset F$. Then for every $s\in K$ such
that $s\ne v$, we
have
\[
g_F(s, t) = \sum_{i\in {I_F}} 1_{s\inv vA_i\cap A_i}(t) = 0 
\]
as $s\inv v\in F$ and $s\inv v \ne e$. 
On the other hand, if $s = v\in K$, then
\[
g_F(v, t) = \sum_{i\in {I_F}} 1_{s\inv vA_i\cap A_i}(t)
= \sum_{i\in {I_F}} 1_{A_i}(t) = 1.
\]
It follows that $g_F |_{K\times \Gamma} = (1_v \ot 1)|_{K\times \Gamma}$ and the
claim follows.
\end{proof}

\section{The Ellis theorem for discrete quantum groups with low duals}

In this section we prove our main result that 
a discrete quantum group with a low dual acts freely 
on its universal compactification. 
As noted in the previous section, the 
classical version of this -- the Ellis theorem -- may be proved using a combinatorial
lemma due to Kat\v etov (Lemma~\ref{lemma:katetov}).
Next we extend Kat\v etov's lemma to the context of
multivalued functions, so that we can apply it in our setting.
We denote the cardinality of a set $X$ by $|X|$.
The proof of the following lemma is
adapted from the proof of Kat\v etov's lemma presented in
\cite[Lemma 3.33]{hindman-strauss}.

\begin{lem} \label{lem:multi-katetov}
  Let $X$ be a set and $f:X\to X$ a multivalued function.
  Suppose that there is a natural number $N$ such that
 $|f(x)|\le N$ and $|f\inv(x)| \le N$ for every $x\in X$.
  Then there is a partition
  \[
  X = \bigsqcup_{i=1}^{2N+1} A_i
  \]
  such that
  \[
  y  \notin f(x) \text{ whenever } x, y \in A_i,\;  x \ne y.
  \]
\end{lem}

\begin{proof}
  Consider functions $g: X\supset D_g \to \{1, 2, \ldots, 2N+1\}$
  such that
  \begin{enumerate}[label=(\roman*)]
  \item $f(D_g) \subset D_g$ (the left-hand side meaning the
    union of the sets $f(x)$ with $x\in D_g$),
  \item for every $x\in D_g$ and $y\in f(x)$, we have $g(y) \ne
    g(x)$ if $y\ne x$. 
  \end{enumerate}
  Collection of such functions is partially ordered by inclusion
  and it follows from Zorn's lemma that this collection has a maximal
  element $g$ (note that the collection is nonempty as the function
  with empty domain is in the collection). 

  We shall show that $D_g = X$ by contradiction. To this end,
  suppose that there is $x\in X\sm D_g$. 
  We shall extend $g$ to a function $h$ such that $h$ satisfies the
  conditions above and contains also $x$ in its domain.
  First, define $h = g$ on $D_g$. Then we define
  $h(x)\in \{1, 2, \ldots, 2N+1\}$ such that $h(x)\ne h(y)$ for every
  such $y\in f(x)\sm \{x\}$ that $h(y)$ is already defined.
  Then $h$ is defined on $D_g \cup \{x\}$ such that also condition
  (ii) is satisfied but we still need to extend this function to make
  sure that also (i) is satisfied. 
  
  Let $f^0(x) = \{x\}$ and for every $k = 0, 1, \ldots$, define
  \[
  f^{k+1}(x) = \bigcup \{ f(y) \mid y \in f^k(x) \}. 
  \]
  Suppose that $h$ is already well-defined on $\bigcup_{k=0}^n f^k(x)$
  and next we extend its definition to $f^{n+1} (x)$. 
  For every $y\in f^{n+1} (x)$, define
  $h(y)\in \{1, 2, \ldots, 2N+1\}$ such that
  $h(y) \ne h(w)$ for every $w\in f(y)\cup f\inv(y)$ such that
  $h(w)$ is already defined and $w\ne y$. This is possible
  since $|f(y)\cup f\inv(y)| \le 2N$ by the assumption on $f$.
  It follows that $h$ is defined on
  \[
  D_g \cup \bigcup_{k=0}^\infty f^k(x)
  \]
  and satisfies (i) and (ii), contradicting the maximality of $g$.
  Hence $D_g = X$ and we obtain the required partition by putting
  $A_i = g\inv(i)$ for every $i = 1, 2, \ldots, 2N+1$.
\end{proof}

\begin{remark}
Suppose that $f$ satisfies the assumptions of the preceding lemma and
in addition $x\notin f(x)$ for any $x$. Then the partition can be made
so that $y\notin f(x)$ whenever $x, y\in A_i$. This formulation is
closer to the Kat\v etov lemma (Lemma \ref{lemma:katetov}),
but is not applicable to our situation.
\end{remark}

Let $\G$ be a compact quantum group. For $ u, v \in \Irr(\G)$ and a finite set 
$F\subset \Irr(\G)$, write
\[
I_{ u\ot v } = \{  w \in \Irr(\G) \mid  w \subset  u\ot v  \}
\]
and
\[
I_{F\ot v } = \bigcup_{ u\in F} I_{ u\ot v }.
\]
We consider these as multisets, that is, sets that allow multiple
instances of the same element. In this case, the multiplicity 
of an element in $I_{ u\ot v }$ is the same as its multiplicity as a
subrepresentation of $ u\ot v$.  

\begin{lem} \label{lem:partition}
Let $F\subset \Irr(\G)$ be a finite set. There is a finite partition
\[
\Irr(\G) = \bigsqcup_i Y_i
\]
such that
\[
I_{F\ot v } \cap I_{F\ot v '} =\emptyset
\]
for any distinct  $v , v '\in Y_i$. 
\end{lem}

\begin{proof}
Define a multivalued function
\[
f: \Irr(\G) \to \Irr(\G), \quad
f(v ) = \{  u\in\Irr(\G)  \mid
           u\subset \conj{ w }\ot  y\ot v  \text{ for some } w ,  y\in F \}.
\]
If $ u\in f(v )$, then $ u\in I_{z\ot v }$ for some irreducible 
$z\subset \conj{ w }\ot  y$ for some $ w ,  y\in F$. 
It follows from \cite[Proposition 2.7.4]{neshveyev-tuset:book} that
$|I_{z\ot v }|$ is bounded by a constant $C_z$, which depends on $z$
but not on $v $.
As $z\subset \conj{ w }\ot  y$ for some $ w,  y\in F$,
the representation $z$ is taken from a fixed finite set.
It follows that $|f(v )|$ is bounded by a constant independent of $v $. 

On the other hand, 
if $ u\in f\inv(v )$, then $v \subset z\ot  u$
for some irreducible $z\subset \conj{ w }\ot  y$ with $ w ,  y\in F$. 
Then, by the Frobenius reciprocity \cite[Theorem~2.2.6]{neshveyev-tuset:book},
$ u\subset \conj{z}\ot v $. Similarly as above, we find an upper
bound for $|f\inv(v )|$ which is independent of $ v$. 

Now we may apply Lemma~\ref{lem:multi-katetov} to the function $f$
and obtain a partition
\[
\Irr(\G) = \bigsqcup_i Y_i
\]
such that
\begin{equation} \label{eq:lemma}
v \notin \{  u\in\Irr(\G)  \mid
    u\subset \conj{ w }\ot  y \ot  v'  \text{ for some } w ,  y \in F\}
\end{equation}
for any distinct  $v , v '\in Y_i$. 
Now suppose that $v , v '\in Y_i$ are distinct and
\[
 u\in I_{F\ot v }\cap I_{F\ot v '}. 
\]
Then $u \subset  w \ot v $ and $ u\subset  y \ot v '$ for
some $ w ,  y \in F$. Hence
$v \subset \conj{ w }\ot u \subset \conj{ w }\ot  y \ot v '$.
This is a  contradiction with~\eqref{eq:lemma}. 
\end{proof}

Recall that a discrete quantum group $\Gd$ is said to have a \emph{low dual}
if there is a universal bound on the dimensions of
the irreducible representations of $\G$, where $\G$ is the dual quantum group of~$\Gd$.

\begin{theorem}
  Every discrete quantum group with a low dual
  acts freely on its universal compactification.
\end{theorem}

\begin{proof}
  Fix a discrete quantum group $\Gd$ with a low dual, and denote the dual quantum group of $\Gd$ by $\G$. 
  Let $ u\in \Irr(\G)$ and let $c_ u\in \ellinf(\Gd)$ be
  supported by $\{ u\}$.  For every $v \in \Irr(\G)$, we have
  \[
  \cop(c_ u)(1\ot p_{\conj v }) = \sum_{ w \in \Irr(\G)}
  \sum_{T:  u\subset  w \ot \conj v } Tc_u p_u T^*
  \]
  where the sum runs through the multiplicities of $ u$ in
  $ w \ot  \conj v $ and   $T: H_ u\to H_ w \ot H_{\conj v }$
  denotes an associated isometric intertwiner. 
  Since $ u\subset  w \ot \conj v $ we have that $ w \subset  u\ot v $ by
  the Frobenius reciprocity
  \cite[Theorem~2.2.6]{neshveyev-tuset:book}.
  As noted in the proof of Lemma~\ref{lem:partition}, the cardinality
  of the set  
  \[
  I_{ u\ot v } = \{  w \in \Irr(\G) \mid  w \subset  u\ot v  \} 
  \]
  is bounded by a constant $K_ u$ depending only on $ u$ and not $v $. 
 We may then write
  \begin{equation} \label{eq:elementary}
  \cop(c_ u)(1\ot p_{\conj v }) = \sum_{k} A_k^{v }\ot B_k^{v }
  \end{equation}
  where   $A_k^v  \in B(H_{ w^v _k})$ and $B_k^{v }\in B(H_{\conj v })$ 
  are such that $\|A_k^ v\| \le \|c_ u\|$ and $\|B_k^ v\| \le 1$.
Since $\Gd$ has a low dual and $|I_{ u\ot v }|\le K_u$, 
the number of summands in \eqref{eq:elementary} 
can be bounded by a universal constant $N$. 
By adding zero terms if necessary, we may assume that the sum runs from $1$ to $N$, for every $ v$.

Let $\F$ be the collection of all finite subsets of
$\Irr(\G)$ and suppose that $F\in \F$ is fixed and $ u\in F$.
Apply Lemma~\ref{lem:partition} to obtain a finite partition
\[
\Irr(\G) = \bigsqcup_i Y_i
\]
such that 
\begin{equation} \label{eq:disjoint}
I_{F\ot v } \cap I_{F\ot v '} =\emptyset
\end{equation}
for any distinct  $v , v '\in Y_i$.
Fix $i$ for now.
In particular, $I_{ u\ot v }\cap I_{ u\ot v '}=\emptyset$
for distinct  $v , v '\in Y_i$,  
so we can define $a_k\in \ellinf(\Gd)$, $k=1, 2, \ldots, N$, by
\[
a_k p_s = \begin{cases}
  A_k^{v } & \text{if $s =  w _k^v $ for some $v \in Y_i$}\\
  0 & \text{otherwise.}
  \end{cases}
\]
Since $\Gd$ has low dual, it is of Kac type by
\cite{kra-sol}.
Hence the maps $j$ are isometries, and we can define $b_k\in \ellinf(\Gd)$, $k=1, 2, \ldots, N$, by
\[
b_k p_s = \begin{cases}
  j(B_k^{v }) & \text{if $s = v $ for $v  \in Y_i$}\\
  0 & \text{otherwise}.
  \end{cases}
\]

Given $w\in F$, we have  
\begin{align*}
  & \sum_{k=1}^N \cop(a_k) (1\ot b_k) (p_w \ot 1)
  = \sum_{k=1}^N \sum_{v \in Y_i} \sum _{T:s\subset w\ot v } 
     Ta_k p_sT^*(1_w\ot j(B_k^{v })).
\end{align*}
Suppose that $a_k p_s\ne 0$ for $s\subset w \ot v $.
Then $s =  w _k^{v '}$  for some $v '\in Y_i$ and $a_k p_s = A_k^{v '}$.
Now $s =  w _k^{v '} \subset u \ot v '$. 
Since $w$ and $ u$ are in $F$, it follows from
\eqref{eq:disjoint} that $v ' = v $. Therefore, 
\begin{equation}\label{eq:decomp}
  \begin{split}
  \sum_{k=1}^N \cop(a_k) (1\ot b_k) (p_w \ot 1)
  &= \sum_{k=1}^N
  \sum_{v \in Y_i} \sum _{T: w _k^v \subset w\ot v } T A_k^{v }
  T^*(1_w\ot j(B_k^{v }))\\
  &= \sum_{v \in Y_i}
  \sum_{k=1}^{n_v}
  \bigl(\cop(A_k^{v })(p_w\ot p_v )\bigr) (1_w\ot j(B_k^{v })).
  \end{split}
\end{equation}

Let $t_v \in \Mor(1, v \ot \conj v )$  as fixed earlier.
By \eqref{eq:j} we have
\begin{align*}
  &\sum_{k=1}^{n_v }
  \Bigl(\bigl(\cop(A_k^{v })(p_w\ot p_v )\bigr)\bigl(1_w\ot  j(B_k^{v })\bigr)\ot 1_{\conj v }\Bigr)
  (1_w\ot t_v )
  =\sum_{k=1}^{n_v } (\cop(A_k^v )(p_w\ot p_v )\ot B_k^{ v
  })(1_w\ot t_v ). 
\end{align*}
Now
\begin{align*}
\sum_{k=1}^{n_v } (\cop(A_k^v )(p_w\ot p_v )\ot B_k^{v })
&= (\cop\ot \id)(\cop(c_ u))(p_w\ot p_v \ot p_{\conj v })  
= (\id\ot \cop)(\cop(c_ u))(p_w\ot p_v \ot p_{\conj v }).
\end{align*}
Write 
\[
\cop(c_ u)(p_w\ot 1) = \sum_\ell C_\ell \ot d_\ell
\]
where $C_\ell\in B(H_w)$, $d_\ell\in \ellinf(\Gd)$ and the sum is finite.
Continuing the earlier calculation we have 
\begin{align*}
  \sum_{k=1}^{n_v }
  \Bigl(\bigl(\cop(A_k^{v })(p_w\ot p_v )\bigr)\bigl(1\ot  j(B_k^{v })\bigr)\ot 1_{\conj v }\Bigr)(1_w\ot t_v ) 
  =  \sum_\ell \sum _{T:s\subset v \ot \conj v }
  (C_\ell \ot T d_\ell p_s T^*)(1_w\ot t_v ).
\end{align*}
Now $T^*t_v \in \Mor(\1, s)$ and $s$ is irreducible, so
$T^*t_v  = 0$ unless $s = \1$. In this case we have
$T = \lambda t_ v$ where $\lambda\in\C$ with $|\lambda|=\|t_ v\|\inv$,
and so $Td_\ell p_\1 T^*t_v  = \epsilon(d_\ell)t_v $.
It follows that
\begin{align*}
& \sum_{k=1}^{n_v }
  \Bigl(\bigl(\cop(A_k^{v })(p_w\ot p_v )\bigr)\bigl(1\ot  j(B_k^{v })\bigr)\ot 1_{\conj v }\Bigr)(1_w\ot t_v ) 
 =  \sum_\ell (C_\ell  \ot \epsilon(d_\ell) 1_v  \ot 1_{\conj v })(1_w\ot t_v )\\
  &\qquad =
  (\id\ot \epsilon)\bigl(\cop(c_ u)\bigr)p_w \ot 1_v  \ot 1_{\conj v })(1_w\ot t_v )
  = (c_u p_w \ot 1_v  \ot 1_{\conj v })(1_w\ot t_v ).
\end{align*}
Since $(1_v \ot s_v ^*)(t_v  \ot 1_v ) = 1_ v$, it follows that
\[
  \sum_{k=1}^{n_v }
  \bigl(\cop(A_k^{v })(p_w\ot p_v )\bigr)\bigl(1_w\ot  j(B_k^{ v
  })\bigr)
  = c_u p_w \ot 1_v .
\]
Inserting this into equation \eqref{eq:decomp}, we have
\[
\sum_{k=1}^N \cop(a_k) (1\ot b_k) (p_w \ot 1)
= \sum_{v \in Y_i} c_u p_w \ot p_v  = (c_u \ot p_{Y_i})(p_w\ot 1)
\]
where $p_{Y_i} = \sum_{v\in Y_i} p_v$. 
Since $w \in F$ is arbitrary, it follows that
\[
\biggl(\sum_{k=1}^N \cop(a_k) (1\ot b_k)\biggr)(f\ot 1)
= (c_u  \ot p_{Y_i})(f\ot 1) 
\]
for every $f$ supported by $F$. Repeating the construction
for each $i$ and denoting the resulting elements by
$a^i_k$ and $b_k^i$, put  
\[
X_F = \sum_i \sum_{k=1}^N \cop(a^i_k) (1\ot b^i_k).
\]
Since the net $(X_F)_{F\in \F}$ is bounded,
it follows from the above that
$X_F\to c_ u\ot 1$ strictly in $M(c_0(\Gd)\ot \ellinf(\Gd))$. 
This in turn implies that 
the linear span of
$\cop(\ellinf(\Gd))(1\ot \ellinf(\Gd))$
is stricly dense in $M(c_0(\Gd)\ot \ellinf(\Gd))$.

\end{proof}


\end{document}